\documentclass[12pt]{article}

\usepackage{graphicx}              
\usepackage{amsmath}               
\usepackage{amsfonts}              
\usepackage{amsthm}                
\usepackage{tikz}
\usepackage{caption}
\usepackage{subcaption}

\newtheorem{thm}{Theorem}[section]
\newtheorem{lem}[thm]{Lemma}

\newtheorem{ques}{Question}[section]
\newtheorem{cor}{Corollary}[section]

\newcommand{\expoi}{{E(A_t(P^{\textrm{Poi}}_n))}}
\newcommand{\exn}{{E(A_t(P^n_n))}}
\newcommand{\exx}{{E(A_t(P^x_n))}}

\def\lr{\left(}
\def\rr{\right)}

\title{The Total Acquisition Number of the Randomly Weighted Path}
\author{Anant Godbole\\
Department of Mathematics \& Statistics \\
East Tennessee State University\\{\tt godbolea@mail.etsu.edu}\and
Elizabeth Kelley\\
Department of Mathematics \\
University of Minnesota\\{\tt ekelley@g.hmc.edu}\and
Emily Kurtz\\
Department of Mathematics \\
Wellesley College\\{\tt ekurtz@wellesley.edu}\and
Pawe\l\ Pra\l at\\
Department of Mathematics \\
Ryerson University\\{\tt pralat@ryerson.ca}\and
Yiguang Zhang\\
Department of Applied Mathematics \& Statistics \\
The Johns Hopkins University\\
{\tt yzhan132@jhu.edu}}
\begin{document}
\maketitle
\smallskip


\begin{abstract}
There exists a significant body of work on determining the acquisition number $a_t(G)$ of various graphs when the vertices of
those graphs are each initially assigned a unit weight. We determine properties of the acquisition number of the path, star, complete, complete bipartite, cycle, and wheel graphs for variations on this initial weighting scheme, with the majority of our work focusing on the expected acquisition number of randomly weighted graphs. In particular, we bound the expected acquisition number $E(a_t(P_n))$ of the $n$-path when $n$ distinguishable ``units" of integral weight, or chips, are randomly distributed across its vertices between $0.242n$ and $0.375n$. With computer support, we improve it by showing that $E(a_t(P_n))$ lies between $0.29523n$ and $0.29576n$. We then use subadditivity to show that the limiting ratio $\lim E(a_t(P_n))/n$ exists, and simulations reveal more exactly what the limiting value equals.  The Hoeffding-Azuma inequality is used to prove that the acquisition number is tightly concentrated around its expected value.  Additionally, in a different context, we offer a non-optimal acquisition protocol algorithm for the randomly weighted path and exactly compute the expected size of the resultant residual set.
\end{abstract}

\section{Introduction}
In this paper, we consider vertex-weighted graphs and denote the weight of vertex $v$ as $w(v)$. Let $G$ be a graph with an initial weight of 1 on each vertex. For adjacent $v,u \in V(G)$, weight can be transferred from $v$ to $u$ via an \emph{acquisition move} if the initial weight on $u$ is at least as great as the weight on $v$. When there are no remaining acquisition moves, the set of vertices with non-zero weight forms an independent set referred to as the \emph{residual set}. The minimal cardinality of this set, $a_t(G)$, is the \emph{total acquisition number} of $G$. A set of acquisition moves that results in a residual set is referred to as an \emph{acquisition protocol} and is \emph{optimal} if the independent set has cardinality $a_t(G).$ This conception of acquisition number was first introduced by \cite{lampertSlater} and has subsequently been investigated in \cite{totalAcquisition, unitAcq}. 

When acquisition moves are allowed to transfer any integral amount of weight from a vertex, the minimum cardinality of the residual set is called the \emph{unit acquisition number}, denoted $a_u(G)$; see \cite{unitAcq}. If acquisition moves are allowed to transfer any non-zero amount of weight from a vertex, the minimum cardinality of the residual set is called the \emph{fractional acquisition number}, denoted $a_f(G)$; see \cite{fractionalAcquisition}. Because we consider only total acquisition number, any instances of the term ``acquisition number" in this paper should be understood to mean ``total acquisition number".

Although we offer a few minor results for graphs with the canonical weighting scheme (where each vertex has initial weight 1), we primarily consider variants on that weighting scheme where the initial weights of vertices are allowed to assume any integral value. Any $n$-vertex graph with vertex labels $\{ 1, 2, \dots, n \}$ can be associated with an integer sequence $(a_1,a_2,\dots,a_n)$ where $a_i$ denotes the initial weight given to vertex $i$. For such \emph{weight sequences} and particular classes of graphs, we consider in Section 2 the size $a_{\textrm{max}}(G)$ of the largest possible residual set, whether $a_t(G)$ changes or remains the same as the unit weight case, and the existence of legal residual sets with sizes equal to every integral value in the interval $[a_t(G),a_{\textrm{max}}(G)]$.  Our primary focus is, however, on the total acquisition number of graphs with randomly weighted vertices.  Although there exists work on random graphs whose vertices each begin with unit weight due to \cite{totalAcquisitionRandom}, we are not aware of any existing work on graphs with randomly-weighted vertices. In Section~2, we make some preliminary remarks. In Section 3, we obtain bounds on the total expected acquisition number of the randomly weighted path, where vertex weighting is assigned according to both the Poisson and Maxwell-Boltzman distributions (in the latter case, the chips are thus considered to be distinguishable, and obviously the Bose-Einstein distribution might yield completely different results!).  We also show that the limiting ratio $\lim E(a_t(P_n))/n$ exists, and that the acquisition number is tightly concentrated around its expected value.  Additionally, in a different context, we offer a non-optimal acquisition protocol algorithm for the randomly weighted path and exactly compute the expected size of the resultant ``residual" set.


\section{Basic results}
In this section, we provide some basic results. The proofs of these results are not very difficult, but the underlying logic is important for Section 3.
\subsection{$a_t(G)=1$}
In this subsection, we consider graphs in which the vertices can have any non-negative integer weight. One question we could ask is what is the smallest maximum vertex weight necessary to drive $a_t(G)$ down to 1? Denote such a value as $smv(G)$. We  specifically consider the complete graph $K_n$ on $n$ vertices; the $n$-cycle and $n$-path $C_n, P_n$; and the star, wheel, and complete bipartite graphs denoted respectively by $K_{1,n}, W_n,$ and $K_{n,m}$. 

It is clear that the $smv(K_n)=smv(W_n)=smv(K_{1,n})=1$, because we can have a special vertex (the center vertex for $W_n$ and $K_{1,n}$; any vertex for $K_n$) absorb the weight of its neighbors first and thus make it the largest weighted vertex in $V(G)$.

For $K_{n,m}$, its $smv$ value is 1 as well. Let $A,B$ be the two vertex sets of $K_{n,m}$, where $|A|\geq |B|$. Let the initial weight of all vertices be 1. Let two vertices $v_A, v_B$ be two arbitrary vertices from $A$ and $B$. At the first step, $v_A$ can acquire all the weight from vertices in $B \backslash \{v_B\}$ and $v_B$ can acquire all the weight from vertices in $A \backslash \{v_A\}$. Because $|A| \geq |B|$, $w(v_B) \geq w(v_A)$ and $v_B$ can acquire the weight from $v_A$. Thus, the $smv$ value for complete bipartite graphs is also 1. We next consider paths and cycles, for which the situation is more nuanced.
\begin{lem}
Let $P_n = v_1v_2 \cdots v_n$ be a path on $n \geq 2$ vertices, where $v_1$ and $v_n$ are the endpoints of the path. Let each vertex have weight at least one.  In order for $v_n$ to be the only vertex in the residual set, its initial weight must be at least $2^{n-2}$.
\end{lem}
\begin{proof}
We prove this lemma by inducting on $n$. The base cases are trivial (simply assign both vertices in $P_2$ an initial weight of 1, and use the configuration 1-1-2 for $P_3$ as the minimal cases; higher initial weights on the non-terminal vertices would simply need us to have more weight on $v_n$). Suppose the statement is true for $P_n$ with initial configuration $1-1-2-\ldots 2^{n-2}$ and let us consider $P_{n+1}$.
After acquiring all the weights from $v_1$ through $v_{n-1}$, $v_n$ has weight at least $2^{n-1}$.  Thus, $v_{n+1}$ must have initial weight at least $2^{n-1}$ as well.
\end{proof}
The problem of finding the smallest maximum vertex weight necessary to drive $a_t(G)$ down to 1 for a path on $n$ vertices is equivalent to finding the smallest initial weight for the middle vertex to absorb all the weight in that path.  By Lemma 2.1, a weight of $2^m$ is needed to take all the weight on an $(m+2)$-path to a leaf vertex; applying this fact to the two middle vertices, we see that a weight of $2^m$ suffices to move all the weight on $P_{2m+4}$ to these vertices and then to one of them.  Solving $n=2m+4$ for $m$, we get $m=n/2-2$ for even $n$.  If $n$ in odd, we get $m=\lceil n/2\rceil-2$ by the same reasoning.  It follows immediately that $smv(C_n)=2^{\lceil n/2 \rceil -2}$ as well.  The $smv$ value for a $m \times n$ grid graph is thus at most $2^{\lceil n/2 \rceil + \lceil m/2 \rceil - 4}$; here we use the strategy of moving all the weight in each row to the center vertex, and then all the weight in the middle column to the center of the grid. For the lower bound, let us note that, regardless where the absorbing vertex $v$ is, at least one of the four corners is at distance $\lceil (m-1)/2 \rceil + \lceil (n-1)/2 \rceil$ from $v$. Hence, at the time when the weight from this corner is pushed to the final destination, the weight at $v$ must be at least $2^{\lceil (m-1)/2 \rceil + \lceil (n-1)/2 \rceil-1}$. However, perhaps some of this weight comes from the other three neighbours of $v$. As a result, we only get that the initial weight at $v$ is at least $2^{\lceil (m-1)/2 \rceil + \lceil (n-1)/2 \rceil-4}$, which is matching the upper bound for $m,n$ both even, and is always by a multiplicative factor of at least $1/4$ away from it. 

\subsection{Size of a residual set}
Let $G=(V,E)$ be an arbitrary graph. Note that the size of the maximum independent set is a natural upper bound of the size of an residual set. By choosing the weight of vertices in $G$ strategically, how many different sizes of residual sets can we get for a given graph? In this subsection, we focus on $K_n, C_n, P_n, W_n,$, and $K_{n,m}$. 

It is clear that the size of the residual set of $K_n$ must be 1 for any assignment of weights, because the maximum independent set of $K_n$ has size 1. Now, consider $C_n$. Because the size of the maximum independent set is $\lfloor \frac{n}{2}\rfloor$, the residual set can be no larger than $\lfloor \frac{n}{2}\rfloor$. Indeed, by assigning the vertices in a largest independent set of $C_n$ the first $\lfloor \frac{n}{2}\rfloor$ largest weights, we can obtain a residual set with $\lfloor \frac{n}{2}\rfloor$ vertices. For example, $\lfloor n/2\rfloor$ 2's and $\lceil n/2\rceil$ 1's can do the job.  Now, to obtain a residual set of size $i$, where $1\leq i \leq \lfloor \frac{n}{2}\rfloor$, we need to choose $i$ vertices which form an independent set that are ``equally spaced to the extent possible", and strategically assign the values of the $i$ largest weights to these vertices so as to enable those vertices to acquire the weights of the other vertices.  This can be achieved because of the reasoning in Section 2.1 by using weights of $1,1, 2,\ldots,2^{\lfloor n/i\rfloor-2}$ on each of the $i$ paths that the cycle can be thought of as being comprised of.  Next we see that the size of the residual set of $W_n$ can be any integer from $1$ to $\lfloor \frac{n}{2}\rfloor$ as well. We can place the smallest weight on the center vertex to reduce the problem to the problem of $C_n$ after the first move. With the same argument, paths can have residual set with size from $1$ to $\lceil \frac{n}{2} \rceil$, where the upper bound is the size of the largest independent set of $P_n$. 

The size of the residual set of $K_{m,n}$ for $m\geq n$ can be any integer from 1 to $m$.  By placing the $m$ largest numbers of the sequence on the vertices of the larger side of the bipartite graph, we ensure that no vertex on the smaller side can acquire any additional chips, thus resulting in a residual set including all vertices of the larger side.  After choosing appropriate weights of vertices, we can ensure that exactly one vertex on the larger side is acquired by assigning the smallest number to one vertex on the larger side, then assigning the next $n$ smallest numbers to the vertices of the smaller side.  Then, as long as the sum of the smallest number and the $m^{\textrm{th}}$ largest number is smaller than the ${m-1}^{th}$ largest number, we have a residual set of size $m-1$.  We can use a similar strategy to get any other number between 1 and $m-2$.

\subsection{Subadditivity} 

Now, we consider $P_n$ with unit weights. Without making use of the fact that $a_t(P_n) = \lceil \frac{n}{4} \rceil$, we can show that: 

\begin{lem}
\label{pathSubadd}
The sequence $\{a_t\left(P_n \right) \}_{n=1}^{\infty}$  is subadditive.
\end{lem}

\begin{proof}
Consider the graph $P_{n+m}$, with $m,n \in \mathbb{Z}$. If $P_{n+m}$ is subdivided into $P_n$ and $P_m$ and distinct acquisition protocols are run on each, then the resulting residual set has size $a_t(P_m)+a_t(P_n).$ Because $a_t(P_{n+m})$ is definitionally the size of the minimal residual set, the fact that it is possible to obtain a residual set of size $a_t(P_n)+a_t(P_m)$ gives the bound
$$a_t(P_{n+m}) \leq a_t(P_n) + a_t(P_m),$$
as desired.
\end{proof}

\begin{cor}
\label{fekete}
For $\{a_t(P_i) \}_{i=1}^{\infty}$, the limit $\lim\limits_{n\rightarrow\infty} \frac{a_t(P_n)}{n}$ exists and is equal to $\inf \frac{a_t(P_n)}{n}$. 
\end{cor}

\begin{proof}
This result follows directly from Lemma 2.2 and Fekete's Lemma. 
\end{proof}

Although it is intuitively obvious that this limit equals $1/4$, the power of subadditivity will become clear in the next section, where we use random weights.

\section{Total Acquisition on Randomly Weighted Graphs}

\subsection{Poisson Distribution}

In this section, we consider the total acquisition number of $P_n$ when each vertex begins with weight $\textrm{Poi}(1)$, i.e.\ the vertices have random weights determined by a sequence of independent Poisson variables with unit mean. We denote this specific configuration as $P_n^{\textrm{Poi}}$. In general, the upper case letter $A$ will be used for the acquisition number when it is viewed as a random variable. We can begin by proving that the limit $$\lim_{n\to\infty}\frac{E(A_t(P^{\textrm{Poi}}_n))}{n}$$ exists, and provide upper and lower bounds for $E(A_t(P^{\textrm{Poi}}_n))$.  We start with a few remarks.

\medskip

\noindent{\bf Remarks.} First, let us note that checking whether a given weighting of $P_n = (v_1, v_2,\ldots, v_n)$ can be
used to move the total weight onto one vertex can be done as follows.  Starting from $v_1$, we
push its weight to the right as much as possible, ending at $v_k$ for some
$k\le n$. Then, independently (and using the initial weighting), we start from $v_n$ and
push its weight to the left as much as possible, ending at $v_\ell$
for some $\ell\ge1$.
It is straightforward to see that if $k\ge\ell-1$, then our task is possible; otherwise, it is not.

Next, finding $a_t(P_n)$ (for a given weighting) can be easily done as follows. Suppose that weights on the subpath $(v_1, v_2,\ldots, v_k)$
can be moved to one vertex. If weights on the subpath $(v_1,v_2 ,\ldots,v_k,v_{k+1})$
can also be moved to one vertex, then this is at least as good strategy as
moving only weighs from $(v_1, v_2 ,\ldots, v_k)$ and then applying the best strategy
for the remaining path (consider simple coupling between the two strategies).
As a result, finding $a_t(P_n)$ for any weighting can be done with
an easy greedy algorithm (and with the support of a computer).

Finally, we get from the above remarks that $E(A_t(P_n^{\textrm{Poi}}))$ is an increasing function of $n$.

\begin{lem}
\label{pathSubadd}
The sequence $\{\expoi\}_{n=1}^\infty$ is subadditive.
\end{lem}

\begin{proof}
Same as the proof of Lemma 2.2 for each sample realization.  We then take expectations to get the result.
\end{proof}

\begin{cor}
\label{fekete}
The limit $\lim\limits_{n\rightarrow\infty} \frac{\expoi}{n}$ exists and is equal to $\inf \frac{\expoi}{n}.$
\end{cor}

\begin{proof}
This result follows directly from Lemma~\ref{pathSubadd} and Fekete's Lemma.
\end{proof}

\begin{thm}
\label{poissonBounds}
The expected acquisition number of $P_{n}^{\textrm{Poi}}$ is bounded as
$$0.242n \leq \expoi \leq 0.375n.$$
\end{thm}

\begin{proof}
Define an {\em island} as the ``clump'' of vertices to the left of the first zero weight; to the right of the last zero weight; or in between any two successive zero weights.  Islands are of non-negative size, and thus each consist of a possibly empty set of non-zero numbers.  The island size thus has a geometric distribution with ``success" probability $1/e$ and expected size  $e-1$, yielding an expected number of $n/e + c; 0 \le c \le1$, for the random number $\Lambda$ of islands. (Note that the expected number of zeros is $n/e$.) Theoretically, the expected total acquisition number could be calculated using the conditional probability expression
\begin{eqnarray*}
\expoi&=& \sum_{j=1}^\Lambda E\left[ a_t(P_{\vert\Lambda_j\vert})\right]\\
&=&\lr\frac{n}{e}+c\rr E[a_t(P_{\vert\Lambda_1\vert})]\\
&=&\lr \frac{n}{e}+c\rr \sum_{j=0}^n E(A_t(P_j^{\rm{Poi}}))P(\vert\Lambda_1\vert=j), 
\end{eqnarray*}
where the second equality follows from Wald's Lemma.  However, calculating $E\left[A_t(P_j^{\rm{Poi}}) \right]$ is difficult for arbitrary $j$. The probability that an island is of size $j$ equals $\lr1-\frac{1}{e}\rr^j\lr\frac{1}{e}\rr$.  Thus, there is a roughly $84\%$ probability that an island has size three or less, and a reasonable lower bound can be obtained by restricting the calculation to those cases. It is clear that $E\left[A_t(P_0^{\rm{Poi}}) \right] = 0$ and that $E\left[A_t(P_1^{\rm{Poi}}) \right] = E\left[A_t(P_2^{\rm{Poi}}) \right] = 1.$
For $P_3$, however, it is possible to have $A_t(P_3) = 1$ or $A_t(P_3) = 2.$ Because $A_t(P_3) = A_t(a-b-c)$ requires that $w(a) > w(b)\ge1$ and $w(c) > w(b)\ge1$, we condition on each of $w(a), w(b)$, and $w(c)$ being non-negative, and, setting $W(b)=k$, we get
\begin{eqnarray*}
P\left[A_t(P_3) = 2 \right] &=& \sum_{k\ge1}\frac{e^{-1}}{1-e^{-1}}\frac{1}{k!}\lr\sum_{j\ge k+1}\frac{e^{-1}}{(1-e^{-1})}\frac{1}{j!}\rr^2\\
&=&\frac{1}{(e-1)^3}\sum_{k\ge1}\frac{1}{k!}\lr\sum_{j\ge k+1}\frac{1}{j!}\rr^2\\
&=&\frac{e^2}{(e-1)^3}\sum_{k=1}^{\infty}\frac{1}{k!}\lr1-\frac{\Gamma(k+1,1)}{\Gamma(k+1)}\rr^2\\
&\approx& 0.10648.
\end{eqnarray*}
Thus $E[A_t(P_3)]$ can be calculated as
\begin{equation}
E[A_t(P_3)] = P\left[A_t(P_3) = 1 \right](1) + P\left[A_t(P_3) = 2 \right](2) \approx 1.10648.
\end{equation}
Using (1) and the monotonicity of $E(A_t(P_n^{\textrm{Poi}}))$, we can now  calculate a lower bound for $\expoi$ as
\begin{align*}
\expoi
&\ge \lr\frac{n}{e}+c\rr \left(\left(1-\frac{1}{e} \right)\frac{1}{e} + \left(1-\frac{1}{e} \right)^2\frac{1}{e} + 1.106\left(1-\frac{1}{e} \right)^3\ \right) \\
& \ge 0.242n.
\end{align*}

To obtain an upper bound, we use the fact that $a_t(P_j) \leq \frac{j+1}{2}$ for any $j$. Returning to our conditional probability expression, this allows us to construct an upper bound for $E(A_t(P_n^{\textrm{Poi}}))$ as 
\begin{align*}
\expoi&=\lr\frac{n}{e}+c\rr\lr\sum_{j=1}^3E(A_t(P_j))P(\vert\Lambda_1\vert=j)+\sum_{j\ge 4}E(A_t(P_j))P(\vert\Lambda_1\vert=j)\rr\\
&\leq \lr\frac{n}{e}+c\rr\sum_{j=1}^3E(A_t(P_j))P(\vert\Lambda_1\vert=j)+\lr\frac{n}{e}+c\rr\sum_{j\ge 4}\frac{j+1}{2}P(\vert\Lambda_1\vert=j)\\
&\le 0.178n+\frac{n}{e} \frac{1-\frac{1}{e}}{2e}\sum_{j=4}^{\infty} j\left(1-\frac{1}{e} \right)^{j-1}+0.029n  \\
&= 0.207n+ \frac{n}{e}\frac{1-\frac{1}{e}}{2e}\frac{(e-1)^3(3+e)}{e^2}  \\
&\approx 0.375n,
\end{align*}
which gives us the desired bounds for $\expoi.$
\end{proof}
The above bounds can certainly be improved, but we do not do so here---rather, we point out methods that might lead to a tightening.  First we can compute $P(A_t(P_4)=2)$ or even more higher order terms so as to improve the lower bound.   For the upper bound, one may do a more careful calculation by using the fact that $A_t(P_j)\le\lceil\frac{j}{2}\rceil$, and separating the argument for $j\ge4$ into the even and odd cases.   However, these methods are likely to yield only incremental improvements, and so we next report on the results of simulations which yield theoretical bounds that are vastly better than the ones above, and also suggest the value of the limiting constant.

\subsection{Simulations}

As we already remarked in Subsection~3.1, with the support of a computer, it is easy to find $a_t(P_j)$ for a given initial weighting. By considering all possible $k^j$ configurations of weights at most $k=k(j)$, we can easily estimate $E(A_t(P_j))$ from below and above. We considered all paths on at most 21 vertices to obtain the following bounds (for more details, see~\cite{program}).

\begin{center}
\begin{tabular}{ | c | c | c | c | }
 \hline
 $j$ & $k=k(j)$ & lower bound & upper bound \\ 
 \hline
3 & 8 & 1.106474556295647 & 1.106485236542439 \\ 
4 & 8 & 1.458146467559788 & 1.458160707876175 \\
5 & 8 & 1.858398253506155 & 1.858424954075593 \\
6 & 8 & 2.117547080007199 & 2.117579120662054 \\
7 & 8 & 2.376630970622960 & 2.376680811599442 \\
8 & 8 & 2.678679193248656 & 2.678736154318619 \\
9 & 7 & 2.990263585826933 & 2.990993168117246 \\
10 & 7 & 3.279939200172999 & 3.280749841020560 \\
11 & 7 & 3.567927403519968 & 3.568997441689286 \\
12 & 6 & 3.857599662110278 & 3.867074161371549 \\
13 & 6 & 4.153795998702126 & 4.165769950472480 \\
14 & 6 & 4.446442617528823 & 4.459336912818369 \\
15 & 5 & 4.680594361425691 & 4.792653655209101 \\
16 & 4 & 4.589627977247097 & 5.299393961081964 \\
17 & 4 & 4.829797279410342 & 5.675769081053113 \\
18 & 4 & 5.066501771940262 & 5.959683875778665 \\
19 & 3 & 3.298233451614096 & 7.696367715681969 \\
20 & 3 & 3.357920459863793 & 7.924328251247145\\
21 & 3 & 3.411118548530028 & 8.613707642327828 \\
 \hline
\end{tabular}
\end{center}

Based on that we get the following.

\begin{cor}
The expected acquisition number of $P_{n}^{\textrm{Poi}}$ is bounded as
$$0.29523n \leq \expoi \leq 0.29576n.$$
\end{cor}

Moreover, we performed a number of experiments on paths of length $n = 100, 000, 000, 000$. Simulations \emph{suggest} that $\expoi \approx 0.295531n$ (again, for more details, see~\cite{program}).

\subsection{dePoissonized model}

Although considering a Poisson model for weight distribution makes it significantly easier to bound $\expoi$, there is typically more interest in models where a fixed amount of weight is distributed on $P_n$.  In order to translate our result for the Poissonized model to this dePoissonized model, we begin by establishing two lemmas.

\begin{lem}
For $P_n$, assigning an initial weight of $\textrm{Poi}(1)$ chips to each vertex is equivalent to considering the model in which we generate the total number of chips according to a ${\rm Poi}(n)$ distribution, and then distribute them independently and uniformly on the $n$ vertices.
\end{lem}
\begin{proof}
One half of the proof follows from the fact that the sum of independent Poi(1) variables has a Poi($n$) distribution.  Next, consider a random distribution of $\textrm{Poi}(n)$ chips on $P_n$ as in the statement of the lemma. The probability that two particular vertices, $u,v$, receive $x,y$ chips respectively (the same argument holds for any number of vertices) is given by
\begin{align*}
P\left[w(u) = x, w(v)=y \right] &= \sum_{r=0}^{\infty} P\left[\sum_{i=1}^n w(v_i) = r \right]P\left[w(u) = x, w(v)=y  \left| \sum_{i=1}^n w(v_i) = r \right. \right] \\
&= \sum_{r=x+y}^{\infty} e^{-n} \frac{n^r}{r!} {r \choose {x,y}} \left(\frac{1}{n} \right)^{x+y}\left(1 - \frac{2}{n} \right)^{r-x-y} \\
&= \frac{e^{-n}}{x!y!} \sum_{r=x+y}^{\infty} \frac{(n-2)^{r-x-y}}{(r-x-y)!} \\
&= \frac{e^{-2}}{x!y!}
\end{align*}
which we recognize as the product of the probability that $w(u) = x; w(v)=y$ if the initial weights on the vertices are determined by an independent  $\textrm{Poi}(1)$ process, as desired.
\end{proof}
The following lemma is critical and valid only for special graphs such as $P_n$:
\begin{lem}
\label{azumaSetup}
Changing the initial weight on a single vertex can change $a_t(P_n)$ by at most $1$.
\end{lem}

\begin{proof}  The proof is an easy consequence of the remarks at the beginning of Section 3.1.
Indeed, after applying the greedy algorithm mentioned there, we decompose $P_n$ into $a_t(P_n)$ subpaths; each subpath has the total acquisition of 1 (with the initial weighting induced on corresponding vertices). Now, changing the initial weight on a single vertex can increase the total acquisition of the corresponding path by 1. Hence, globally, $a_t(P_n)$ can increase by at most 1. This finishes the proof as it also implies that it cannot decrease by more than 1 (if it decreases by more than that, then after switching back to the original weighting, the parameter increases by more than one).
\end{proof}

Even though the main intent of the use of Lemma 3.4 is dePoissonization, it also quickly gives a very sharp concentration of $A_t(P_n^{\textrm{Poi}})$ around $\expoi$. 

\begin{thm}
For $A_t(P_n^{\textrm{Poi}})$ determined by a series of random unit Poisson trials, $X_1, \dots, X_n$ and any $\phi(n)\to\infty$,
\begin{align*}
\textrm{Pr}\left[ \left| A_t\left(P_n^{\textrm{Poi}}  \right) - \expoi    \right|  > \sqrt{2n\phi(n)} \right] \rightarrow 0
\end{align*}
as $n \rightarrow \infty$ and $A_t\left(P_n^{\textrm{Poi}} \right)$ is therefore tightly concentrated in an interval of width $O(\sqrt{n\phi(n)})$ around $\expoi=\Theta(n)$.
\end{thm}

\begin{proof}
\label{expectedConcentration}
From Lemma~\ref{azumaSetup}, we know that for every $i$ and any two sequences of possible outcomes ${x_1, \dots, x_n}$ and ${x_1, \dots, x_{i-1},x'_i,x_{i+1}\ldots,x_n}$,
\[
\left|\lr A_t(P_n^{\textrm{Poi}})\vert X_1 = x_1, \dots, X_n= x_n \rr - \lr A_t(P_n^{\textrm{Poi}})\vert X_j= x_j\ (j\ne i), X_i = x_i' \rr \right| \leq 1.
\]
It follows from the Hoeffding-Azuma inequality that
\[
\textrm{Pr}\left[ \left| A_t\left(P_n^{\textrm{Poi}}  \right) - E\left(A_t\left(P_n^{\textrm{Poi}}  \right)  \right)\right|  > \sqrt{2n\phi(n)} \right] \le
2e^{-\phi(n)}\to0,
\]as desired.
\end{proof}

We will let $\exn$ and $\exx$ respectively denote the expected total acquisition number when $n$ and $x$ tokens are randomly placed on $P_n$.
\begin{lem}
For all $x \in \left[n - \phi(n)\sqrt{n},n+\phi(n)\sqrt{n} \right]$, where $\phi(n)\to\infty$ is arbitrary,
\begin{align*}
\exn - \phi(n)\sqrt{n} \leq \exx \leq \exn+ \phi(n)\sqrt{n}
\end{align*}
\end{lem}

\begin{proof}
This follows immediately from Lemma 3.4, and, moreover, holds for the random variable $A_t(P_n^n)$ as well, before expectations are taken.
\end{proof}

Together, these lemmas can be used to show that, when $n$ is sufficiently large, the bounds from Theorem~\ref{poissonBounds} also apply to the dePoissonized model.

\begin{thm}
For the dePoissonized chip distribution process on $P_n$,
\begin{align*}
\lim_{n\rightarrow\infty} \frac{\exn}{n} = \lim_{n\rightarrow\infty} \frac{\expoi}{n}.
\end{align*}
That is, the limit both exists and is identical to the limit for the Poissonized model.
\end{thm}

\begin{proof}
By Chebychev's inequality, with probability at most $1-1/\phi^2(n)$ we have $|x-n|\le\sqrt{n}\phi(n)$ if $x\sim{\rm Poi}(n)$, so that by Lemma 3.6
\[\vert A_t(P_n^n)-A_t(P_n^x)\vert\le\sqrt{n}\phi(n).\]  On the other hand, if $|x-n|>\sqrt{n}\phi(n)$, then trivially 
\[\vert A_t(P_n^n)-A_t(P_n^x)\vert\le n.\]
Combining the above two facts, we see that for any fixed $\phi(n)=o(\sqrt{n})$ such that $\phi(n) \to \infty$ as $n \to \infty$, and with $B_n=[n -\phi(n)\sqrt{n}, n +\phi(n)\sqrt{n}]$
\begin{align}
\frac{\expoi}{n} &= \sum_{x \geq 0} \frac{e^{-n}n^x}{x!}\frac{\exx}{n}\nonumber \\
&= \sum_{x\in B_n} \frac{e^{-n}n^x}{x!}\frac{\exx}{n} +  \sum_{x \not\in B_n} \frac{e^{-n}n^x}{x!}\frac{\exx}{n} \\
&\le\frac{\exn}{n}+\frac{\phi(n)}{\sqrt n}+\frac{1}{\phi^2(n)}\lr\frac{\exn}{n}+1\rr \nonumber \\
& = (1+o(1)) \frac{\exn}{n},
\end{align}
as $\expoi = \Omega(n)$.
Likewise by just including the first term in (2), we get that
\begin{equation}
\frac{\expoi} {n} \ge (1-o(1))\lr\frac{\exn}{n}-\frac{\phi(n)}{\sqrt{n}}\rr = (1-o(1)) \frac{\exn}{n}.\end{equation}
Inequalities (3) and (4) prove the result.
\end{proof}

\subsection{Uniform Distribution}
So far, our discussion has focused primarily on optimal acquisition protocols. For small examples or particularly simple graphs, it is often possible to definitively determine the optimal acquisition protocol. In larger or more complicated cases, however, doing so becomes laborious and complexity issues become more relevant. In order to sidestep this issue, we shift in this section to considering the size of the residual set produced by an algorithmic acquisition protocol. We will consider the process which is defined in Theorem~\ref{yiguangAlg}, which offers an algorithmic acquisition protocol for instances of the randomly weighted path where each vertex has a unique weight. In order to motivate the use of this algorithm, however, we first state below one of the main results of [2], namely that if a total of $t\gg n^5$ chips are distributed on $P_n$ using a uniform random process, the probability that two or more vertices have the same initial weight goes to zero as $t \rightarrow \infty$. (If $t\ll n^5$ this probability is tending to 1 as $n \to \infty$.)  This result exhibits a scenario under which the conditions of Theorem 3.9 are satisfied.

\begin{lem}
\label{unique}
Let a total of $t$ chips, where $t \gg n^5$, be distributed on $P_n$ using a uniform random process and let $X=X(n)$ denote the number of pairs of vertices that receive the same number of chips. Then,
\begin{align*}
\lim_{n \rightarrow \infty} Pr\left[X \geq 1 \right] = 0.
\end{align*}
\end{lem}

In what follows, we assume that the path is weighted such that vertex weights are distinct. According to Lemma 3.8, this may be realized with high probability by placing, e.g.,  $t=t_n\gg n^5$ tokens randomly on $P_n$ but we do not use this fact explicitly or implicitly.

\begin{thm}
\label{yiguangAlg}
Let $P_n$ be weighted by using a random permutation of $n$ distinct integers $w_1, \dots, w_n$. If each acquisition move consists of the vertex with the highest weight receiving the weight of its immediate neighbors, then the expected acquisition number $A_{td}$ satisfies
$$\lim_{n\rightarrow \infty}\frac{E[A_{td}(P_n)]}{n} = \frac{e^2-1}{2e^2}.$$
\end{thm}

\begin{proof}
Let the chips be randomly distributed on $P_n$. Let $w_k(i)$, for $1 \leq i \leq n$, denote the weight of vertex $i$ after the $k$th step of the algorithm. Initially, each vertex is equally likely to have the largest weight. Without loss of generality, suppose $w_0(j)> w_0(i)$ for all $i<j$ and $i>j$. If $j$ has two neighbors, then at the end of the first step $w_1(j)= w_0(j-1)+w_0(j)+w_0(j+1)$. If $j$ has one neighbor, then $w_1(j) = w_0(j)+w_0(j-1)$ or $w_1(j) = w_0(j)+w_0(j+1)$. After the first step, vertex $j$ cannot acquire the weight of any other vertices in the path because all its neighbors have zero weight.  Therefore, calculating the acquisition number of $P_n$ reduces to calculating the acquisition numbers of the resulting smaller path or paths, as shown in Figure~\ref{fig:algReduction}, since
$$a_{td}(P_n) = a_{td}(P^{'}) + a_{td}(P^{''}) + 1.$$ 

\begin{figure}
\centering
	\begin{subfigure}[b]{0.4\textwidth}
	\includegraphics[scale=1]{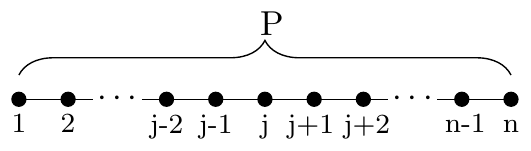}
	\end{subfigure}
	\begin{subfigure}[b]{0.4\textwidth}
	\includegraphics[scale=1]{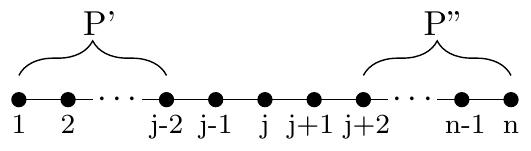}
	\end{subfigure}
	\caption{An illustration of $P_n$ before (left) and after (right) the first step of the algorithm. If vertex $j$ has the highest weight, then calculating the acquisition number of $P$ reduces to calculating the acquisition numbers of $P'$ and $P''$.}
	\label{fig:algReduction}
\end{figure}

We therefore obtain the identity
$$E[A_{td}(P_n)]= 1 + \frac{1}{n}\left(\sum_{i=0}^{n-2}E[A_{td}(P_i)] + E[A_{td}(P_{n-2-i})]\right) = 1 + \frac{2}{n}\sum_{i=0}^{n-2}E[A_{td}(P_i)].$$
Let $G(x)=\sum_{i\geq0}E[A_{td}(P_i)]x^i$ be the generating function for $E[A_{td}(P_n)]$. In order to obtain a closed form expression for $G(x)$, we apply standard generating function techniques to calculate:
\begin{align*}
nE[A_{td}(P_n)] = n + 2\sum_{i=0}^{n-2}E[A_{td}(P_i)] \\
xG'(x)=\frac{x}{(1-x)^2}+2G(x)\frac{x^2}{1-x}.
\end{align*}
Solving the resulting differential equation, we find
\begin{align*}
G(x)=\frac{1+2Ce^{-2x}}{2-4x+2x^2}=\frac{Ce^{-2x}}{(x-1)^2}+\frac{1}{2(x-1)^2},
\end{align*}
where $C$ is an unknown constant. We would like to find an explicit formula for $E[A_{td}(P_n)]$. To do so, we use known generating series representations to note that
\begin{align*}
\frac{Ce^{-2x}}{(x-1)^2}=C \left(\sum_{i=0}^n\frac{(-2)^i}{i!}x^i\right)\left(\sum_{i=0}^n(j+1)x^j\right),
\end{align*}
and equate the $x^n$ coefficients to obtain
\begin{align*}
E[A_{td}(P_n)] = C\sum_{i=0}^n(n+1-i)\frac{(-2)^i}{i!} + \frac{n+1}{2}.
\end{align*}
We then use the fact that $A_{td}(P_2) = 1$ to determine that $C= -1/2$. Therefore,
\begin{align*}
E[A_{td}(P_n)] = -\frac{1}{2}\sum_{i=0}^n(n+1-i)\frac{(-2)^i}{i!} + \frac{n+1}{2}
\end{align*}
and
\begin{align*}
\lim_{n\rightarrow \infty}\frac{E[A_{td}(P_n)]}{n} =\frac{e^2-1}{2e^2}.
\end{align*}
This finishes the proof of the result.
\end{proof}
Assuming that each vertex has unique weight, we can similarly apply this algorithm to cycles, stars, wheels, and complete bipartite graphs. After a single iteration of the algorithm, the cycle graph $C_n$ reduces to $P_{n-3}$ and therefore 
\begin{align*}
E[A_{td}(C_n )]= 1 + E[A_{td}(P_{n-3})].
\end{align*}
 For the star graph $S_n$, there are two possible cases. Let $v_c$ denote the center vertex of $S_n$. If $v_c$ has the highest weight, then during the first iteration it will acquire the weight of all other vertices and the acquisition number will trivially be $1$. If $v_c$ is not the vertex with highest weight, then during the first iteration one of the leaves will acquire the weight on $v_c$, leaving $n-1$ unconnected vertices and giving the graph the acquisition number $n-1$. Thus, we have
\begin{align*}
E[A_{td}(S_n)] = \frac{(n-1)^2}{n} + \frac{1}{n} = \frac{n^2-2n}{n}.
\end{align*}
For the wheel graph $W_n$, there are again two cases. Again, let $v_c$ denote the center vertex of $W_n$. If $v_c$ has the highest weight, then it will similarly acquire the weight of all other vertices, producing an acquisition number of 1. If $v_c$ is not the vertex with highest weight, then after the first iteration $W_n$ will reduce to $C_n$. It follows that we have
\begin{align*}
E[A_{td}(W_n)] &=   \frac{1}{n} + \frac{n-1}{n}\left(1 + E[A_{td}(P_{n-3})] \right)
\end{align*}
Finally, let us consider the complete bipartite graph $K_{n,m} = (U,V,E)$. Let $v^*$ denote the vertex with highest weight. If $v^* \in U,$ then during the first iteration of the algorithm $v^*$ will acquire the weight of all vertices in $V$, leaving $n$ vertices in $U$ totally disconnected and producing an acquisition number of $n$. If $v^* \in V,$ then during the first iteration $v^*$ will similarly acquire the weight of all vertices in $U$, producing an acquisition number of $m$. Thus, we have
\begin{align*}
E[A_{td}(K_{n,m})] = \frac{n^2 + m^2}{n+m}.
\end{align*}
Similarly, the multipartite graph $K_{n_1,\dots,n_k}$ has 
\begin{align*}
E[A_{td}(K_{n1,\dots,n_k})] = \frac{\sum_{i=1}^k n_i^2}{\sum_{i=1}^k n_i}.
\end{align*}

\section{Open questions}

Finally, we can conclude by offering some open questions that arose during our investigation of randomly-weighted graphs.

\begin{ques}
If $t$ chips are randomly distributed, what are the expected total acquisition numbers of $K_{n,n}$, $K_{n,m}$, $K_{1,m}$, $L_n$, and $G_{m,n}$?
\end{ques}

\begin{ques}{(Diameter 2 graph)}
For any graph $G$ with diameter two, it's known that ${a_t(G) \leq 32 \ln n \ln \ln n}$ \cite{totalAcquisition} and conjectured that ${a_t(G) \leq c}$ (where perhaps $c=2$). If we instead randomly distribute $n$ chips on a graph with diameter two, can $a_t(G)$ be similarly bounded? What if we randomly distribute $t$ chips?
\end{ques}

\begin{ques}
In our acquisition protocol algorithm, each acquisition move involves transferring weight to the highest weighted vertex. Can an acquisition protocol algorithm where acquisition moves are confined to transferring weight away from the vertex with lowest weight be well defined? If so, what is the expected size of the resultant residual set?
\label{ques:algProtocol}
\end{ques}

\begin{ques}
For the canonical acquisition problem, where each vertex of $G$ begins with weight 1, there exists an \emph{acquisition game} variant where two players, Max and Min, make alternate acquisition moves in an attempt to, respectively, maximize and minimize the size of the residual set. The \emph{game acquisition number}, $a_g(G)$, is defined as the size of the residual set under optimal play. Similar investigations could be done on the game acquisition number of randomly weighted graphs.
\end{ques}

In our preliminary investigation into Question~\ref{ques:algProtocol}, we found that the major difficulty in designing such an algorithm is in ensuring that it is well-defined. Even if all vertices of $P_n$ initially have unique weight, it is possible that an acquisition move could result in two vertices subsequently having equal weight. In such a case, it is not clear which vertex should transfer its weight in the next step of the algorithm. Any well-defined algorithm must be able to handle this case without influencing the size of the resulting residual set.

\%bibliographystyle{plainnat}

\end{document}